%% file: fraccliquesagain.tex
\documentclass{article}
\usepackage{amsmath,setspace,scalefnt}
\usepackage[usenames,dvipsnames,svgnames,table]{xcolor}
\usepackage{amsfonts,amsthm}
\usepackage{amssymb,enumitem,verbatim}
\usepackage{graphicx}
\usepackage{tikz,caption,subcaption}
\usepackage[normalem]{ulem}
\usepackage[dvipsnames]{xcolor}

\usepackage[pdfstartview={XYZ null null null}]{hyperref} 
\usepackage{pstricks,tikz}
\usetikzlibrary{positioning,chains,fit,shapes,calc, fit, decorations, decorations.markings}

\newcommand{\changed}[1]{{#1}}
\newcommand{\changedout}[1]{{{}}}

\newtheorem{theorem}{Theorem}[section]
\newtheorem{lemma}[theorem]{Lemma}

\newtheorem{corollary}[theorem]{Corollary}
\newtheorem{conjecture}[theorem]{Conjecture}
\newtheorem{claim}[theorem]{Claim}

\newtheorem{definition}[theorem]{Definition}

\numberwithin{equation}{section}

\newcommand{\PP}{\mathbb{P}}
\newcommand{\N}{\mathbb{N}}

\newcommand{\KK}{\mathcal{K}}
\newcommand{\FF}{\mathcal{F}}

\newcommand{\MM}{\mathcal{M}}

\renewcommand{\d}{\delta}

\newif\ifcutfinished
\cutfinishedfalse

\renewcommand{\epsilon}{\varepsilon}
\newcommand{\e}{\epsilon}
\renewcommand{\subset}{\subseteq}

\newcommand{\COMMENT}[1]{}
\renewcommand{\COMMENT}{\footnote} 

\title{Fractional clique decompositions of dense graphs}

\author{Richard Montgomery\footnote{School of Mathematics,
University of Birmingham,
Edgbaston,
Birmingham,
B15 2TT,
UK. r.h.montgomery@bham.ac.uk}}
\date{}

\begin{document}
 \maketitle

\begin{abstract} For each $r\ge 4$, we show that any graph~$G$ with minimum degree at least $(1-1/(100r))|G|$ has a fractional $K_r$-decomposition. This improves the best previous bounds on the minimum degree required to guarantee a fractional $K_r$-decomposition given by Dukes (for small $r$) and Barber, K\"uhn, Lo, Montgomery and Osthus (for large $r$), giving the first bound that is tight up to the constant multiple of~$r$ (seen, for example, by considering Tur\'an graphs).

In combination with work by Glock, K\"uhn, Lo, Montgomery and Osthus, this shows that, for any graph $F$ with chromatic number $\chi(F)\ge 4$, and any $\e>0$, any sufficiently large graph $G$ with minimum degree at least $(1-1/(100\chi(F))+\e)|G|$ has, subject to some further simple necessary divisibility conditions, an (exact) $F$-decomposition.
\end{abstract}


\section{Introduction}
Given a graph $F$, a graph $G$ has an $F$-decomposition if we can cover the edges of $G$ exactly with edge-disjoint copies of $F$. The study of such decompositions dates back at least as far as 1847, when Kirkman~\cite{Kirkman} showed that any complete graph $K_n$ on~$n$ vertices has a triangle ($K_3$-)decomposition if and only if $n\equiv 1\text{ or }3\mod 6$. As the edges in a triangle decomposition of $K_n$ are partitioned into triangles, we must have $3|e(K_n)$. Furthermore, the $n-1$ neighbours of each vertex $v$ are divided into pairs forming triangles with $v$, so we must have $2|n-1$. In combination, this gives rise to the necessary \changed{and sufficient}
 condition on $n$ for the existence of a triangle decomposition of $K_n$.

It was not until the 1970's that Kirkman's theorem was generalised by Wilson~\cite{wilson1,wilson2,wilson3,wilson} to consider $F$-decompositions of cliques for more general graphs $F$. If a graph $G$ has an $F$-decomposition, then it follows immediately from the partition of $E(G)$ into copies of $F$ that $e(F)$ divides $e(G)$. Furthermore, by considering the copies of $F$ used in the decomposition around each vertex, we can see that the highest common factor of the degrees of the vertices in $F$ must divide the degree of each vertex in $G$. If $G$ satisfies both of these necessary conditions, then we say $G$ is $F$-divisible. For each graph~$F$, Wilson~\cite{wilson1,wilson2,wilson3,wilson} showed that any sufficiently large $F$-divisible complete graph has an $F$-decomposition. In 2014, the long-conjectured generalisation of Wilson's theorem to hypergraph cliques $F$ was proved in a breakthrough result of Keevash~\cite{keevash}. Very recently, an independent, combinatorial, proof of this generalisation was given by Glock, K\"uhn, Lo and Osthus~\cite{GKLO}, who then built on their methods to give a full generalisation of Wilson's theorem to arbitrary hypergraphs $F$~\cite{GKLO17}.

In general, we do not expect a simple characterisation of the $F$-divisible graphs which have an $F$-decomposition. For example, Dor and Rasi~\cite{npcomplete} have shown that determining whether a graph has an $F$-decomposition is NP-complete if $F$ has a connected component with at least 3 edges. It is therefore natural to ask instead if all sufficiently dense $F$-divisible graphs have an $F$-decomposition. Much of the work in this area has been motivated by the following beautiful conjecture of Nash-Williams on triangle decompositions.

\begin{conjecture}[Nash-Williams~\cite{NashWilliams}]\label{conjNW}
There exists $N\in \N$ such that, if~$G$ is a $K_3$-divisible graph on at least $N$ vertices with $\delta(G)\geq 3|G|/4$, then $G$ has a $K_3$-decomposition.
\end{conjecture}

The first progress towards Conjecture~\ref{conjNW} was given by Gustavsson~\cite{gustavsson}, who showed, for each graph $F$, that there is some constant $\e(F)>0$ such that every sufficiently large $F$-divisible graph $G$ satisfying $\delta(G)\geq (1-\e(F))|G|$ has an $F$-decomposition. The bound on $\e(F)$ claimed by Gustavsson left much room for improvement, showing only that $\e(F)\leq 10^{-37}|F|^{-94}$.
In recent years, a rich seam of progress was initiated by Barber, K\"uhn, Lo and Osthus~\cite{BKLO} by relating $F$-decompositions in dense graphs to \emph{fractional $F$-decompositions}.

Given a graph~$F$, we say that a graph $G$ has a \emph{fractional $F$-decomposition} if we may (non-negatively) weight the copies of $F$ in $G$ so that each edge in $E(G)$ is in copies of $F$ with total weight 1. That is, letting $\mathcal{F}(G)$ be the set of all copies of $F$ in $G$, there is some function $\omega:\FF(G)\to [0,1]$ so that, for each $e\in E(G)$, we have
\[
\sum_{F'\in \FF(G):e\in E(F')}\omega(F')=1.
\]
In 2014, Barber, K\"uhn, Lo and Osthus~\cite{BKLO} introduced an innovative iterative absorption method capable of turning an \emph{approximate} $F$-decomposition in an $F$-divisible graph~$G$ with high minimum degree into an $F$-decomposition, where an approximate $F$-decomposition is a disjoint set of copies of~$F$ in~$G$ covering most of the edges of~$G$.  Haxell and R\"odl~\cite{haxellrodl} had demonstrated that a large graph with a fractional $F$-decomposition must have an approximate $F$-decomposition. This allowed Barber, K\"uhn, Lo and Osthus~\cite{BKLO} to give much improved bounds on $\e(F)$ by using bounds on the minimum degree required to guarantee an appropriate fractional $F$-decomposition.

The methods in~\cite{BKLO} were subsequently developed and extended by Glock, K\"uhn, Lo, Osthus and the current author~\cite{GKLMO}.
In order to describe this accumulated progress, for each graph $F$ and integer $n$, let $\delta^*_F(n)$ be the least $\d>0$ such that any graph $G$ with $n$ vertices and $\delta(G)> \delta n$ has a fractional $F$-decomposition. For each graph $F$, let $\delta^*_F=\limsup_{n\to\infty}\delta_F^*(n)$. The main results in~\cite{GKLMO} imply the following.

\begin{theorem}[Glock, K\"uhn, Lo, Montgomery, Osthus~\cite{GKLMO}]\label{GKLMOthm} Let $F$ be a graph, let $\e>0$, and let $\chi=\chi(F)$. Any sufficiently large $F$-divisible graph $G$ with $\delta(G)\geq (\max\{\delta^*_{K_{\chi}},1-1/(\chi+1)\}+\e)|G|$ has an $F$-decomposition.
\end{theorem}

Aside from its own intrinsic interest, determining the value of $\delta^*_{K_r}$ for each $r$ therefore has a significant part to play in the study of decompositions of dense graphs. In particular, to prove Conjecture~\ref{conjNW} asymptotically it is sufficient to show that $\delta^*_{K_3}\leq 3/4$ (as already followed from the results in~\cite{BKLO}).

In the particular case of triangles, increasingly good bounds on $\delta^*_{K_3}$ were given by Yuster~\cite{yuster2005asymptotically}, Dukes~\cite{dukes,dukes2} and Garaschuk~\cite{garaschuk2014linear}, before Dross~\cite{dross} gave an elegantly efficient proof that $\delta^*_{K_3}\leq 0.9$. For each $r\geq 4$, Yuster~\cite{yuster2005asymptotically} showed that $\delta^*_{K_r}\leq 1-1/(9r^{10})$ and gave a construction showing that $\delta^*_{K_r}\geq (1-1/(r+1))n$. Dukes~\cite{dukes,dukes2} used tools from linear algebra to show that $\delta^*_{K_r}\leq 1-2/(9r^{2}(r-1)^2)$, before Barber, K\"uhn, Lo, Osthus and the current author~\cite{BKLMO} were able to generalise and extend Dross's methods for fractional triangle decompositions to show that $\delta^*_{K_r}\leq 1-1/(10^4r^{3/2})$.

In this paper, we show that, for each $r\geq 4$, $\delta^*_{K_r}\leq 1-1/(100r)$. This improves the known upper bound for $\delta^*_{K_r}$ for each $r\geq 4$, and confirms, up to the constant $100$, the correct dependence of $\delta^*_{K_r}$ on $r$.

\begin{theorem}\label{maintheorem}
Let $r\geq 4$. If a graph $G$ has minimum degree at least $(1-1/(100r))|G|$, then $G$ has a fractional $K_r$-decomposition.
\end{theorem}

In combination with Theorem~\ref{GKLMOthm}, this implies the following.

\begin{corollary}\label{whee} Let $F$ be a graph with $\chi(F)\geq 4$ and let $\e>0$. Any sufficiently large $F$-divisible graph $G$ with $\delta(G)\geq (1-1/(100\chi(F))+\e)|G|$ has an $F$-decomposition.
\end{corollary}

Consideration of Tur\'{a}n graphs with $\chi(F)-1$ classes (for example) confirms Corollary~\ref{whee} is tight up to the constant that appears before $\chi(F)$. We have not sought further small improvements in the corresponding constant in Theorem~\ref{maintheorem}, where they would complicate the proof for little gain.
Substantial new ideas appear needed to approach the conjectured value of $\delta^*_{K_r}$.

At the very highest level, our methods to prove Theorem~\ref{maintheorem} take the same form as those used in~\cite{BKLMO}. We find an initial weighting of some cliques in $G$ which is \emph{close} to a fractional $K_r$-decomposition, before making local adjustments to correct this to a fractional $K_r$-decomposition. Within this framework, however, our methods are entirely different. We introduce a simple and efficient way to make the adjustments to the initial weighting, and use a much improved initial weighting of the cliques in $G$ (in comparison to the simple, (essentially) uniform, initial weighting used in~\cite{BKLMO}).

In particular, we find the following simple observation useful: A graph $G$ has a fractional $K_r$-decomposition if we can randomly pick an $r$-clique in $G$ (with some well-chosen probability distribution) so that the probability an edge $e$ is in this $r$-clique is constant across all the edges $e\in E(G)$. This observation, seen by using a weighting on the $r$-cliques proportional to the probability distribution, is trivial. However, the perspective it brings allows us to use a random process to select a random $r$-clique of~$G$, before translating this to an initial weighting which would be more difficult to conceive or describe directly.

The rest of this paper is organised as follows. After describing our notation, we give a more detailed sketch of our methods in Section~\ref{proofsketch}. In Section~\ref{initialweighting}, we find an initial weighting of the cliques in our graph, before, in Section~\ref{sec:correction}, correcting this to give a fractional $K_r$-decomposition. Aside from the elementary, but slightly involved, calculation required to prove a key lemma (Lemma~\ref{minusmatching}), which we defer to Section~\ref{tedium}, Sections~\ref{initialweighting} and~\ref{sec:correction} give the simplest exposition of our methods for finding a fractional $K_r$-decomposition using a minimum degree bound with, up to a constant, the correct dependence on $r$ (proving Lemma~\ref{maintheoremweakened}). In Section~\ref{improvements}, we then make some improvements to a central lemma in Section~\ref{initialweighting} (Lemma~\ref{fraccliqueapprox1}) to reduce the bounds used by our methods, and hence prove Theorem~\ref{maintheorem}. In Section~\ref{tedium}, we give the calculation required to prove Lemma~\ref{minusmatching}.

\subsection{Notation}
A graph $G$ has vertex set $V(G)$, edge set $E(G)$, and minimum degree $\delta(G)$. For each $r\ge 2$, we denote by $\KK_r(G)$ the set of copies of $K_r$, the clique with $r$ vertices, in $G$. Where the graph used is clear from context we will use $\KK_r=\KK_r(G)$. For a set $E\subset V(G)^{(2)}$, we denote by $G+E$ and $G-E$ the graphs with vertex set $V(G)$ and edge sets $E(G)\cup E$ and $E(G)\setminus E$ respectively. In particular, we often use this when $E$ is a matching, i.e.\ a set of independent edges. For each $e\in V(G)^{(2)}$ we let $G-e=G-\{e\}$ and $G+e=G+\{e\}$.

For a graph $G$, we denote by $\bar{G}$ the \emph{complement} of $G$ \changed{--} the graph with vertex set $V(G)$ and edge set $V(G)^{(2)}\setminus E(G)$. For each $v\in V(G)$, $N(v)$ is the set of neighbours of $v$, and $N^c(v)=V(G)\setminus N(v)$, the set of non-neighbours of $v$. When we have a weighting $w_K$, $K\in \KK_r(G)$, of the $r$-cliques in a graph $G$, we say for each edge $e\in E(G)$ that the \emph{weight over $e$} is $\sum_{K\in \KK_r(G):e\in E(K)}w_K$. Given any event $A$, we let
\[
\mathbf{1}_{A}=\left\{\begin{array}{ll} 1 &\text{ if $A$ occurs,} \\ 0 &\text{ otherwise.}\end{array}\right.
\]


\section{Proof Sketch}~\label{proofsketch}
In order to aid our sketch, let us recap very briefly the methods used by Barber, \changed{K\"uhn, Lo,} Osthus and the current author in~\cite{BKLMO}, where methods originated by Dross~\cite{dross} were extended and generalised. In order to find a fractional $K_r$-decomposition of a graph $G$ with high minimum degree in~\cite{BKLMO}, an initial (essentially uniform) weighting was given to the subgraphs of $G$ isomorphic to $K_r$, before a series of small local changes to this weighting were made (using structures called `gadgets') to correct this weighting to a fractional $K_r$-decomposition. In contrast, here we make our initial weighting using a random process which is capable of getting far closer to a fractional decomposition than a simple uniform weighting of cliques (see Section~\ref{sketch:initial}). We use this, in fact, to get close to a fractional $K_{2r+2}$-decomposition of the graph $G$. We then make our corrections `within the $(2r+2)$-cliques' to convert this to a fractional $K_r$-decomposition of $G$ (see Section~\ref{sketch:correction}).

For each part of the proof, it is crucial that we can show that any complete graph with at least $2r+2$ vertices has a \changed{fractional} $K_r$-decomposition, even if we remove an arbitrary matching. Such a clique with removed edges is sufficiently symmetric that we may give a fractional $K_r$-decomposition directly, but as the calculation is slightly involved we defer it to Section~\ref{tedium}. We state the required result here, however, for reference.

\begin{lemma}\label{minusmatching} Let $r\geq 3$ and $k\geq 2r+2$. If $M\subset E(K_{k})$ is a matching, then $K_{k}-M$ has a fractional $K_r$-decomposition.
\end{lemma}

We will give our first initial weighting in Section~\ref{initialweighting}, before returning to it in Section~\ref{improvements} to make improvements. We do this in order to give, in Sections~\ref{initialweighting} and~\ref{sec:correction}, as clean as possible an exposition of our methods to demonstrate there is some $\e>0$ such that, for all $r\ge 3$, any graph~$G$ with minimum degree $(1-\e/r)|G|$ has a fractional $K_r$-decomposition (see Lemma~\ref{maintheoremweakened}).

In order to find our initial weighting we gain from (variations of) the following simple observation.
If we can pick a random $r$-clique $H$ from a graph $G$ so that each edge in $G$ is equally likely to appear in $E(H)$, then $G$ has a fractional $K_r$-decomposition. Indeed, the weights $w_K=e(G)\cdot \mathbb{P}(H=K)/\binom{r}{2}$, $K\in \KK_r$, will form a fractional $K_r$-decomposition of $G$.

We will use a random process to weight subgraphs $H\subset G$ with $\delta(H)\geq |H|-6$ so that each edge is given total weight close to 1 (that is, for each edge the sum of the weight of the subgraphs containing that edge is close to 1). If each such subgraph $H$ has at least $32r+62$ vertices, then we can find a fractional $K_{2r+2}$-decomposition of $H$ by partitioning the non-edges of $H$ into 5 matchings (using a little extra structure in our particular graphs $H$) and repeatedly applying Lemma~\ref{minusmatching}. We then improve our methods in Section~\ref{improvements}, using more structure that naturally arises in the subgraphs $H$ to reduce their size to only $18r+18$ vertices. The method for initially weighting these subgraphs $H$ remains the same and we will now give a representative sketch of this process, picking a random subgraph with $50r$ vertices.

\subsection{Our initial weighting}\label{sketch:initial}
Let $G$ be a graph with $50rm$ vertices and $\delta(G)\geq (1-1/(100r))|G|\geq |G|-m/2$. We will sketch how to pick a random subgraph $H$ of $G$ with $50r$ vertices so that $\delta(H)\geq |H|-6$ and every edge is approximately equally likely to appear in $H$ (c.f.\ Lemma~\ref{pickrandomsubgraph}). We do so by describing a random process which chooses a subgraph $H$ with $|H|=50r$ and $\delta(H)\geq |H|-3$ in which \changed{every vertex is equally likely to appear} in $H$, before describing how to alter this process so that each edge is approximately equally likely to appear in $H$ (at the expense of weakening the degree condition to $\delta (H)\geq |H|-6$).

We choose the vertices $\{a_1,\ldots, a_{50r}\}$ of $H$ randomly in $50r$ stages. At stage $i$, we choose a subset $A_i$ of $m$ vertices to \emph{consider}, from vertices in $G$ that have not yet been considered, and then pick a vertex $a_i\in A_i$ uniformly at random. By including in $A_i$ all the non-neighbours of $a_{i-1}$ that have not yet been considered (of which there are at most $m/2$), we ensure that every non-neighbour of $a_{i-1}$ is considered at stage $i$ or earlier. Thus, each vertex $a_i$ has no non-neighbours in $V(H)$ except for itself and possibly $a_{i-1}$ or $a_{i+1}$, and therefore the final graph $H$ satisfies $\delta(H)\geq |H|-3$. Precisely, we carry out the following process for each $i$, $1\leq i\leq 50r$.
\begin{itemize}
\item Pick $A_i\subset V(G)\setminus (\bigcup_{j<i}A_j)$ uniformly at random subject to $|A_i|=m$ and $N^c(a_{i-1})\setminus (\bigcup_{j<i}A_j)\subset A_i$ (possible as $|N^c(a_{i-1})|\leq m/2$).
\item Pick $a_i\in A_i$ uniformly at random.
\end{itemize}
Each vertex is considered exactly once in this process, and added to $V(H)$ with probability $1/m$. If the two vertices in an edge $e$ are considered in different stages, then the probability that $e$ appears in $H$ is $1/m^2$. If an edge appears within some set $A_i$, then it cannot appear in $H$. Some edges in~$G$ may be much more likely than others to appear within some set $A_i$, and thus be less likely to appear in $H$. The resulting variation in how likely each edge is to appear in $H$ can (for some graphs $G$) be too much to later correct using our methods, and therefore we need alter this random process.

Consider then taking twice as many vertices in $A_i$ (that is, $2m$ vertices) at each stage (for only $25r$ stages in total), and picking \emph{two} vertices $a_i$ and $b_i$ from $A_i$ to add to $V(H)$. By including any unconsidered non-neighbours of $a_{i-1}$ or $b_{i-1}$ in $A_i$, we can ensure $H$ is still almost complete, in fact satisfying $\delta(H)\geq |H|-6$. An edge that appears in some set $A_i$ may now possibly appear in $H$, an improvement on the first process. However, this probability is too small (close to $1/2m^2$) compared to edges whose vertices are considered at different stages (still $1/m^2$).

Fortunately, an edge is only significantly more likely than others to appear in $A_i$ if it appears in $(N^c(a_{i-1})\cup N^c(b_{i-1}))\setminus (\bigcup_{j<i}A_j)\subset A_i$ (as the other vertices in $A_i$ are chosen randomly from the remaining unconsidered vertices). As this set is a subset of $A_i$ with size at most $m=|A_i|/2$, we will be able to choose $a_i$ and $b_i$ from $A_i$ so that any edges in $(N^c(a_{i-1})\cup N^c(b_{i-1}))\setminus (\bigcup_{j<i}A_j)$ appear in $H$ with probability $1/m^2$, but so that each vertex in $A_i$ still appears in $V(H)$ with probability $1/m$ (otherwise we will alter the probability an edge appears in $H$ if its vertices are considered at different stages). Thus, we will be able to choose our subgraph $H$ so that each edge is in $H$ with roughly the same probability.

The precise process we use is given in the proof of Lemma~\ref{pickrandomsubgraph}, and is depicted in Figure~\ref{processpicture}.


\subsection{Correcting the weighting}\label{sketch:correction}
Suppose we have a graph $G$ along with a weighting $w_K$, $K\in \KK_{2r+2}$, of the $(2r+2)$-cliques in $G$ so that each edge in $G$ has total weight at least $1$ and at most $1+1/(2r)$. (Such will be the result of our initial weighting.) We will convert this into a fractional decomposition of $G$ into subgraphs with $2r+2$ vertices which are complete except for some independent non-edges -- say the set of such subgraphs in $G$ is $\MM$. As each of the subgraphs in $\MM$ has a fractional $K_r$-decomposition by Lemma~\ref{minusmatching}, we can combine fractional $K_r$-decompositions of the subgraphs in $\MM$, and the fractional decomposition of $G$ into subgraphs in $\MM$, to get a fractional $K_r$-decomposition of~$G$.

Suppose then an edge $e\in E(G)$ is given weight $1+\lambda$ by the weighting of $(2r+2)$-cliques, with $0\leq \lambda\leq 1/2r$. Consider what would happen if, for each clique $K\in \KK_{\changed{2r+2}}$ containing $e$, we replaced the weight $w_K$ on $K$ by $w_K/(1+\lambda)$ and added weight $(1-1/(1+\lambda))w_K$ to $K-e$.  We get a weighting of the graphs in $\MM$ such that the weight on $e$ is 1, and the weight on the other edges is unchanged, while only weight $(1-1/(1+\lambda))w_K\leq w_K/(2r+1)$ has been moved from any clique $K$ containing $e$.

For each $K\in \KK_{\changed{2r+2}}$, we could similarly adjust the weight on any independent set of edges $M$ in $E(K)$ simultaneously by moving weight from $K$ to $K-M$. Furthermore, by moving weight from $K$ to $K-E$ for different subsets $E\subset M$, we can make different corrections to the weight on different edges in $M$. As each such clique $K\in \KK_{\changed{2r+2}}$ can be covered by $2r+1$ sets of independent edges, we will see that we can make controlled adjustments of up to $w_K/(2r+1)$ to the weight on the edges in $K$ without decreasing the weight on $K$ to become negative. Making appropriate such adjustments for each clique in $K\in \KK_{\changed{2r+2}}$ will allow us to decrease the weight on each edge until it is exactly 1 by moving weight off the cliques $K\in \KK_{\changed{2r+2}}$ and onto other subgraphs in $\MM$. This is carried out in Section~\ref{sec:correction}, to prove Lemma~\ref{correction}, which in combination with the initial weighting lemma (Lemma~\ref{fraccliqueapprox1}, and its improvement Lemma~\ref{fraccliqueapprox2}) proves Theorem~\ref{maintheorem}.


\section{Choosing an almost-complete random subgraph}\label{initialweighting}
In this section we will show that, given some minimum degree condition in a graph $G$, we can find a weighting $w_K$, $K\in \KK_r$, so that, for each $e\in E(G)$, $1-1/r\leq\sum_{K\in \KK_r:e\in E(K)}w_K\leq 1$. As sketched in Section~\ref{proofsketch}, we will use a random process to pick a subgraph $H$ in $G$ with $\delta(H)\geq |H|-6$ so that every edge in $G$ is roughly equally likely to appear within $H$ (proving Lemma~\ref{pickrandomsubgraph}). Each subgraph $H$ will be large enough that we can show $H$ has a fractional $K_r$-decomposition, and we can combine this with the weighting from the probability distribution on such graphs $H$ to get the required approximate fractional $K_r$-decomposition of $G$.

In fact, the random graph $H$ will have the stronger property that it contains a spanning subgraph $H'$ which is isomorphic to a graph of the following form.

\begin{definition}\label{Mrdefn} For each integer $r\geq 1$, let $M_r$ be the graph with vertex set $\{a_1,b_1,\ldots,a_r,b_r\}$ which has every edge present between sets $\{a_i,b_i\}$ and $\{a_j,b_j\}$ if $i\notin \{j-1,j,j+1\}$, and no other edges.
\end{definition}

Note that, for each $r\geq 3$, $\delta(M_r)=2r-6$. Using Lemma~\ref{minusmatching} repeatedly, we can show that any graph $H$ with $32r+62$ vertices which contains a copy of $M_{16r+31}$ has a fractional $K_r$-decomposition (see Lemma~\ref{fraccliqueapprox1}). Later, in Section~\ref{improvements}, we will carry out some more work to allow us to use graphs $H$ with only $18r+18$ vertices (see Lemma~\ref{fraccliqueapprox2}).

\subsection{Our initial probability distribution}
We are now ready to give our random process. Note that the following lemma will be applied eventually with $r$ replaced by a larger function of~$r$. We will also use a divisibility condition on the number of vertices in $G$ in this lemma; this, as we shall see, we can assume by duplicating vertices in our initial graph.

\begin{lemma}\label{pickrandomsubgraph} Let $r\geq 3$ and let $G$ be a graph with $n=2rm$ vertices and $\delta(G)\geq n-m/2$. Let $\mathcal{M}$ be the set of induced subgraphs of $G$ with $2r$ vertices which contain a copy of $M_r$. Then, with an appropriate probability distribution, we may randomly select a graph $M\in \MM$ so that, for each $e\in E(G)$,
\begin{equation}\label{edgeaim}
1-\frac{4}{r}\leq m^2\cdot\mathbb{P}(e\in E(M))\leq 1.
\end{equation}
\ifcutfinished \qed
\else
\fi
\end{lemma}
\ifcutfinished
\else
\begin{proof}
Let $B=\{a_1,b_1,\ldots,a_r,b_r\}$ be a random subset of $V(G)$ picked according to the following method, which is depicted in Figure~\ref{processpicture}: Let $B_1=\emptyset$ and carry out the following \changed{steps} for $1\leq i\leq r$.
\begin{enumerate}[label = \changed{\textbf{S\arabic{enumi}}}]
\item \label{step0} If $i>1$, let $B_i= (N^c(a_{i-1})\cup N^c(b_{i-1}))\setminus (\bigcup_{j<i}A_j)$, noting that $|B_i|\leq m$.
\item Let $A_{i,1}$ be the set $B_i$ combined with an $(m-|B_i|)$-sized subset of $V(G)\setminus ((\bigcup_{j<i}A_j)\cup B_i)$ selected independently and uniformly at random.\label{varies1}
\item \label{step1} Let $A_{i,2}$ be an $m$-sized subset of $V(G)\setminus ((\bigcup_{j<i}A_j)\cup A_{i,1})$ selected independently and uniformly at random, and let $A_i=A_{i,1}\cup A_{i,2}$.
\item \label{varies2} Let $\{a_i,b_i\}$ be a pair of distinct vertices from $A_i$ chosen independently at random so that
\begin{equation}\label{edgedist}
\mathbb{P}(\{a_i,b_i\}=\{a,b\})=\left\{\begin{array}{ll}
1/m^2 & \text{ if }\{a,b\}\subset A_{i,1}\text{ or }\{a,b\}\subset A_{i,2} \\
1/m^3 & \text{ otherwise.}\end{array}\right.
\end{equation}
\end{enumerate}
Let $M=G[B]$. We will show that $M$ is a random subgraph in $\MM$ which satisfies~\eqref{edgeaim} for each $e\in E(G)$.

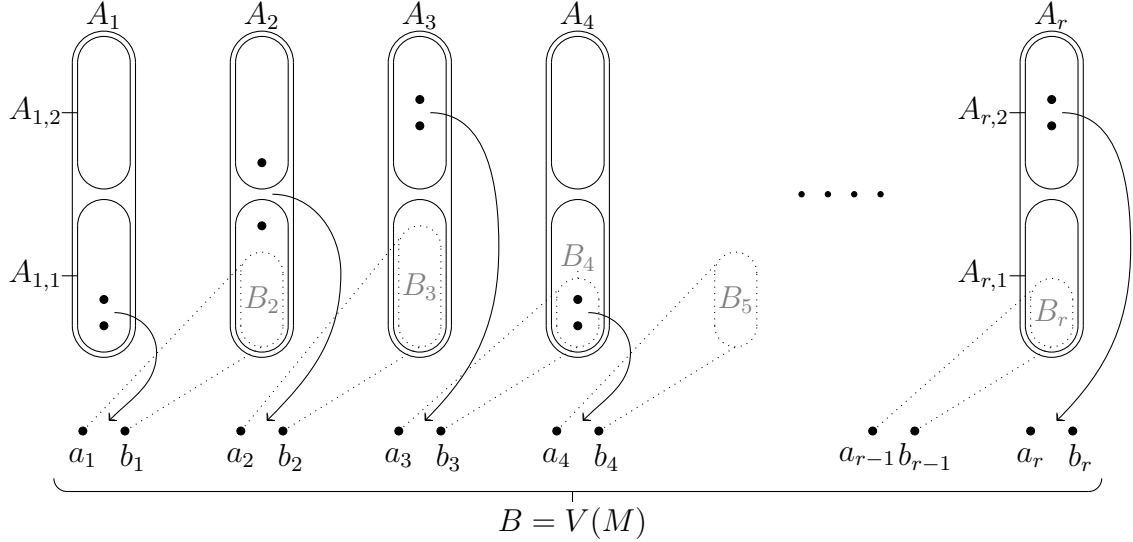
\begin{figure}

\begin{center}
\input{processpicture}
\end{center}

\caption{The random process used to select the set $B=\{a_1,b_1,\ldots,a_r,b_r\}$ which forms the vertex set of the random subgraph~$M$. At each stage $i$, $1\leq i\leq r$, a set of vertices $A_i=A_{i,1}\cup A_{i,2}$ is considered, from which we select two vertices $a_i$ and $b_i$. For $i\geq 2$, $A_{i,1}$ is chosen to contain any non-neighbours of $a_{i-1}$ and $b_{i-1}$ which have not yet been considered (the vertices in the set $B_i$).}\label{processpicture}
\end{figure}

In the process above, we found a random partition $(A_i)_{i=1}^r$, and selected 2 vertices from each set $A_i$ to form $B$, so that $|B|=2r$. Note further that, for each $i<r$, $N^c(a_i)\cup N^c(b_i)\subset B_{i+1}\cup (\bigcup_{j\leq i}A_{j})$. Therefore, for each $i<r$ and $j>i+1$, $a_j,b_j\in A_j$ are neighbours of $a_i$ and $b_i$. Thus, $M\in \MM$.

Therefore, it is only left to show that~\eqref{edgeaim} holds for each $e\in E(G)$. We will show this by conditioning on how the vertices of $e$ appear in the random partition.

Fix an edge $e=uv\in E(G)$, and note that if $u$ appears in $A_i$, but $v$ has not appeared in $A_j$, $j\leq i$, then the probability that $e\in E(M)$ is $1/m^2$. Indeed, in this case $u$ and $v$ were selected to be in $B$ independently at random, with (by symmetry) probability $1/m$, so that
\begin{equation}\label{conditioning1}
\mathbb{P}\big(uv\in E(M)|\exists i\text{ s.t. }u\in A_i\text{ and $v\notin \bigcup\nolimits_{j\leq i}A_j$}\big)=1/m^2.
\end{equation}
This also holds with $u$ and $v$ switched, so we only need consider the case when $u$ and $v$ first appear in the same set $A_i$. In this case, the edge $uv$ appears in $M$ with the distribution in~\eqref{edgedist}, so that
\begin{equation}\label{conditioning2}
\mathbb{P}(uv\in E(M)|\exists i\text{ and }j\text{ s.t.\ }\{u,v\}\subset A_{i,j})=1/m^2,
\end{equation}
and
\begin{equation}\label{conditioning3}
\mathbb{P}(uv\in E(M)|\exists i\text{ s.t.\ }u\in A_{i,1}\text{ and }v\in A_{i,2})=1/m^3.
\end{equation}
By~\eqref{conditioning1},~\eqref{conditioning2} and~\eqref{conditioning3}, and the symmetric cases with $u$ switched with $v$, we therefore have
\begin{equation*}\label{lunch1}
\mathbb{P}(e\in E(M)|\nexists i\text{ s.t. } V(e)\cap A_{i,1},V(e)\cap A_{i,2}\neq \emptyset)=1/m^2
\end{equation*}
and
\begin{equation*}\label{lunch2}
\mathbb{P}(e\in E(M)|\exists i\text{ s.t. } V(e)\cap A_{i,1},V(e)\cap A_{i,2}\neq \emptyset)=1/m^3.
\end{equation*}
Thus,
\changed{
\begin{align*}\label{lunch2}
\mathbb{P}(e\in E(M))&= \mathbb{P}(\nexists i\text{ s.t. } V(e)\cap A_{i,1},V(e)\cap A_{i,2}\neq \emptyset)\cdot \frac{1}{m^2}+ \mathbb{P}(\exists i\text{ s.t. } V(e)\cap A_{i,1},V(e)\cap A_{i,2}\neq \emptyset)\cdot \frac{1}{m^3}\\
& =\frac{1}{m^2}-\left(\frac{1}{m^2}-\frac{1}{m^3}\right)\cdot \mathbb{P}(\exists i\text{ s.t. } V(e)\cap A_{i,1},V(e)\cap A_{i,2}\neq \emptyset),
\end{align*}
so that}
\begin{equation}\label{almostdone}
1\ge m^2\cdot \mathbb{P}(e\in E(M))\ge 1-\mathbb{P}(\exists i\text{ s.t. } V(e)\cap A_{i,1},V(e)\cap A_{i,2}\neq \emptyset).
\end{equation}
We will show the following claim.

\begin{claim}\label{claim1}
For each $i\in[r]$, $\mathbb{P}(v\in A_{i,2}|u\in A_{i,1})\leq 2/r$.
\end{claim}

Claim~\ref{claim1} implies that~\eqref{edgeaim} holds for the edge $uv$, as required. Indeed, we have
\begin{align*}
\mathbb{P}(\exists i\text{ s.t.\ }v\in A_{i,2}\text{ and }u\in A_{i,1})
&=\sum_{i=1}^r\mathbb{P}(v\in A_{i,2}|u\in A_{i,1})\changed{\cdot} \mathbb{P}(u\in A_{i,1})
\\
&\leq (2/r)\changed{\cdot}\sum_{i=1}^r\mathbb{P}(u\in A_{i,1})\leq 2/r.
\end{align*}
In combination with the symmetric case with $u$ switched with $v$, and~\eqref{almostdone}, this implies~\eqref{edgeaim}. Therefore, it is left only to prove Claim~\ref{claim1}.

\medskip

\noindent
\emph{Proof of Claim~\ref{claim1}.}
Firstly, if $i\leq 3r/4$, we have
\[
\mathbb{P}(v\in A_{i,2}|u\in A_{i,1})\leq \mathbb{P}\big(v\in A_{i,2}|\changed{u\in A_{i,1}\land} v\notin \big(\bigcup\nolimits_{j< i} A_j\big)\cup A_{i,1}\big)=\frac{m}{2(r-i)m+m}\leq \frac{2}{r}.
\]
We therefore can assume that \changed{$i> 3r/4$}. It is plausible that $\mathbb{P}(v\in A_{i,2}|u\in A_{i,1})\leq 2/r$, as, if $u\in A_{i,1}$, then in the above process it seems far more likely that $v$ appears in $\bigcup_{j<i}(A_{j}\setminus B_j)$ than in $A_{i,2}$, as the former set is at least $(i-1)$ times as large as the latter. Conditioning on $u\in A_{i,1}$, however, might reduce the probability $v$ appears in the former set, \changed{so we will prove carefully that $\mathbb{P}(v\in A_{i,2}|u\in A_{i,1})\leq 2/r$.

We need to consider the random process up to and including the choice of $A_{i,2}$ (when we will certainly know whether $v\in A_{i,2}$ or not). In particular, we consider the space of all the possible choices that can be made during this part of the process. That is, all possible choices of
\[
X:=(A_{1,1},A_{1,2},a_1,b_1,\ldots, A_{i-1,1},A_{i-1,2},a_{i-1},b_{i-1},A_{i,1},A_{i,2}).
\]
Note that
\begin{equation}\label{ylabel} Y=(\bar{A}_{1,1},\bar{A}_{1,2},\bar{a}_1,\bar{b}_1,\ldots, \bar{A}_{i-1,1},\bar{A}_{i-1,2},\bar{a}_{i-1},\bar{b}_{i-1},\bar{A}_{i,1},\bar{A}_{i,2})
\end{equation}
is a possible value for $X$ if
\begin{itemize}
\item $\bar{A}_{1,1},\bar{A}_{1,2},\ldots,\bar{A}_{i,1},\bar{A}_{i,2}$ are disjoint sets with size $m$ in $V(G)$,
\item for each $j\in [i-1]$, $\bar{a}_j,\bar{b}_j\in \bar{A}_{j,1}\cup \bar{A}_{j,2}$, and
\item for each $j\in [i-1]$, $N^c(\bar{a}_j)\cup N^c(\bar{b}_j)\subset (\bigcup_{j'<j}(\bar{A}_{j',1}\cup \bar{A}_{j',2}))\cup \bar{A}_{j,1}$.
\end{itemize}
To calculate $\mathbb{P}(v\in A_{i,2}|u\in A_{i,1})$ we need to consider all possible such $Y$ for which, in addition, $u\in \bar{A}_{i,1}$ -- let $\mathcal{Z}$ be the set of such sequences. Furthermore, let $\mathcal{Z}'$ be the set of such sequences in $\mathcal{Z}$ for which $v\in \bar{A}_{i,2}$. We therefore have
\[
\mathbb{P}(v\in A_{i,2}|u\in A_{i,1})=\frac{\PP(X\in \mathcal{Z}')}{\PP(X\in \mathcal{Z})}.
\]

Calculating this fraction is complicated by the fact that $\PP(X=Y)$ is likely to differ across $Y\in\mathcal{Z}$. Thus, let us partition $\mathcal{Z}$ into classes $\mathcal{Z}_1,\ldots,\mathcal{Z}_\ell$, for the smallest possible $\ell$, so that there are $p_1,\ldots p_k>0$ such that, for each $k\in [\ell]$ and $Y\in \mathcal{Z}_k$, $\PP(X=Y)=p_k$.
For each $k\in [\ell]$, let $\mathcal{Z}'_k=\mathcal{Z}_k\cap \mathcal{Z}'$. We will show the following claim.

\begin{claim}\label{claim2} For each $k\in [\ell]$, $|\mathcal{Z}'_k|\leq 2|\mathcal{Z}_k|/r$.
\end{claim}
If Claim~\ref{claim2} holds, then
\[
\mathbb{P}(v\in A_{i,2}|u\in A_{i,1})=\frac{\PP(X\in \mathcal{Z}')}{\PP(X\in \mathcal{Z})}=\frac{\sum_{k\in [\ell]}p_k|\mathcal{Z}'_k|}{\sum_{k\in [\ell]}p_k|\mathcal{Z}_k|}\leq \max_{k\in [\ell]}\left\{\frac{|\mathcal{Z}'_k|}{|\mathcal{Z}_k|}\right\}\leq \frac{2}{r},
\]
and hence Claim~\ref{claim1} holds. It is sufficient then to prove Claim~\ref{claim2}.

\medskip

\noindent\emph{Proof of Claim~\ref{claim2}.} Fix $k\in [\ell]$. Create an auxilliary bipartite graph $H$ with vertex classes $\mathcal{Z}_k\setminus \mathcal{Z}'_k$ and $\mathcal{Z}'_k$, where $YY'$ is an edge for $Y\in \mathcal{Z}_k\setminus \mathcal{Z}'_k$ and $Y'\in \mathcal{Z}'_k$ exactly when $Y$ can be transformed into $Y'$ by switching $v$ with some other vertex.

Note that, for each $Y\in \mathcal{Z}_k\setminus \mathcal{Z}'_k$ labelled as in~\eqref{ylabel}, $v$ has to be switched with some vertex in $\bar{A}_{i,2}$ in order to get a sequence in $\mathcal{Z}'_k$. Thus, for each $Y\in \mathcal{Z}_k\setminus \mathcal{Z}'_k$, $d_H(Y)\leq m$. We will show, for each $Y'\in \mathcal{Z}'_k$, $d_H(Y')\geq mr/2$, whence
\[
(mr/2)\cdot|\mathcal{Z}'_k| \leq \sum_{Y'\in \mathcal{Z}'_k}d_H(Y')=e(H)=\sum_{Y\in \mathcal{Z}_k\setminus \mathcal{Z}'_k}d_H(Y)\leq m|\mathcal{Z}_k\setminus \mathcal{Z}'_k|\leq m|\mathcal{Z}_k|,
\]
and thus Claim~\ref{claim2} follows.

Fix then $Y'=(\bar{A}'_{1,1},\bar{A}'_{1,2},\bar{a}'_1,\bar{b}'_1,\ldots, \bar{A}'_{i-1,1},\bar{A}'_{i-1,2},\bar{a}'_{i-1},\bar{b}'_{i-1},\bar{A}'_{i,1},\bar{A}'_{i,2})\in  \mathcal{Z}'_k$.
Recall the steps \ref{step0}--\ref{varies2} in the random process. Note that \ref{step0} is deterministic and, for each $j\leq i$, the probability of each possible choice at \ref{step1} is always the same. For each $j\leq i$, the probability of each possible choice at \ref{varies1} depends exactly on the size of
 $(N^c({a}_j)\cup N^c({b}_j))\setminus (\bigcup_{j'<j}{A}_{j'})$. For each $j< i$, the probability of each possible choice at \ref{varies2} depends on whether each of $a_j$ and $b_j$ appears in $A_{j,1}$ or $A_{j,2}$.

Thus, noting that, for each $j<i$, $\{\bar{a}'_j,\bar{b}'_j\}\subset N^c(\bar{a}'_j)\cup N^c(\bar{b}'_j)$, we can switch any vertices not in $\{u\}\cup \big(\bigcup_{j<i}( N^c(\bar{a}'_j)\cup N^c(\bar{b}'_j)\big)$ in $Y'$ and get a sequence with the same probability of occurence as $Y$.
Therefore, if we take $Y'$ and switch $v$ with any vertex in $\big(\bigcup_{j<i}(\bar{A}'_{j,1}\cup \bar{A}'_{j,2})\big)\setminus \big(\bigcup_{j<i}( N^c(\bar{a}'_j)\cup N^c(\bar{b}'_j)\big)$ then we get a sequence in $\mathcal{Z}_k\setminus \mathcal{Z}'_k$. Therefore,
\[
d_H(Y')\geq \Big|\big(\bigcup_{j<i}(\bar{A}'_{j,1}\cup \bar{A}'_{j,2})\big)\big\backslash \big(\bigcup_{j<i}(N^c(\bar{a}'_j)\cup N^c(\bar{b}'_j)\big)\Big|\geq 2(i-1)m-2(i-1)m/2= (i-1)m\geq mr/2,
\]
as required. This completes the proof of Claim~\ref{claim2}, and hence the lemma.}\hspace{4.5cm} \qed
\end{proof}
\fi


\subsection{Turning to our initial weighting}
We will now show our (appropriately-sized) random subgraph $H$ has a fractional $K_r$-decomposition using the following lemma, before later improving our methods in Section~\ref{improvements}.

\begin{lemma}\label{minusmatchingind} Let $r\geq 3$ and $\ell\ge 1$.
Let $G$ be a graph on at least $2^\ell r+2^{\ell +1}-2$ vertices so that $E(\bar{G})$ can be split into $\ell $ matchings. Then, $G$ has a fractional $K_r$-decomposition.
\ifcutfinished \qed
\else
\fi
\end{lemma}
\ifcutfinished
\else
\begin{proof} We will prove the lemma for all $r\ge 3$ using induction on $\ell $. When $\ell =1$, the lemma follows directly from Lemma~\ref{minusmatching}. Let $\ell >1$ and assume then that the lemma holds for $\ell -1$, and let $G$ be a graph on at least $2^\ell r+2^{\ell +1}-2=2^{\ell-1}(2r+2)+2^\ell-2$ vertices so that $E(\bar{G})$ can be split into $\ell $ matchings.

Let $E(\bar{G})=E_1\cup E_2$, where $E_1$ is a matching and $E_2$ can be split into $\ell -1$ matchings. Let $\mathcal{M}$ be the set of induced subgraphs of $G$ with $2r+2$ vertices and \changed{no two vertices in the same edge in $E_2$}. By the inductive hypothesis, $G+E_1$ has a fractional $K_{2r+2}$-decomposition. Each copy of $K_{2r+2}$ in $G+E_1$ has the same vertex set as a graph in $\mathcal{M}$, and therefore $G$ has a fractional decomposition into graphs in $\mathcal{M}$. By Lemma~\ref{minusmatching}, each subgraph in $\mathcal{M}$ has a fractional $K_r$-decomposition, and thus $G$ has as well.
\end{proof}
\fi
We can now combine Lemmas~\ref{pickrandomsubgraph} and~\ref{minusmatchingind} to give our initial weighting.
\begin{lemma}\label{fraccliqueapprox1} Let $r\geq 3$ and let $G$ be a graph with $n=(32r+62)m$ vertices and $\delta(G)\geq n-m/2$. Then, there is a set of weights $w_K$, $K\in \KK_r$, such that, for each edge $e\in E(G)$,
\begin{equation}\label{approxbound}
1-\frac 1r\leq \sum_{K\in \KK_r:e\in E(K)}w_K\leq 1.
\end{equation}
\end{lemma}
\begin{proof} Let $\mathcal{M}$ be the set of induced subgraphs of $G$ with $32r+62$ vertices which contain a copy of $M_{16r+31}$. By Lemma~\ref{pickrandomsubgraph}, we may find non-negative weights $p_M$, $M\in \MM$, so that, for each $e\in E(G)$,
\begin{equation}\label{edgeaim21}
1-\frac 1r\leq 1-\frac{4}{16r+31}\leq \sum_{M\in \MM:e\in E(M)}p_M\leq 1.
\end{equation}
The set $E(\bar{M}_{16r+31})$ can be covered by $5$ matchings, and hence, for each $M\in \MM$, $E(\bar{M})$ can as well. Therefore, by Lemma~\ref{minusmatchingind}, each $M\in \MM$ has a fractional $K_r$-decomposition, so there exist non-negative weights $w_{M,K}$, $K\in \KK_r$, such that, for each $e\in E(G)$,
\[
\sum_{K\in \KK_r:e\in E(K)}w_{M,K}=\mathbf{1}_{\{e\in E(M)\}}.
\]
For each $K\in \KK_r$, let $w_K=\sum_{M\in \MM}p_M\cdot w_{M,K}\ge 0$. Then, for each $e\in E(G)$,
\begin{equation}\label{obvious}
\sum_{K\in \KK_r:e\in E(K)}w_K=\sum_{M\in \MM}p_M\cdot\sum_{K\in \KK_r:e\in E(K)}w_{M,K}=\sum_{M\in \MM}p_M\cdot \mathbf{1}_{\{e\in E(M)\}}.
\end{equation}
Combining~\eqref{edgeaim21} and~\eqref{obvious} shows that the weights $w_K$, $K\in \KK_r$, satisfy~\eqref{approxbound}.
\end{proof}


\section{Tidying up the approximate fractional decomposition}\label{sec:correction}
The aim of this section is to prove the following lemma, which shows that a graph with an approximate fractional $K_{2r+2}$-decomposition has a fractional $K_r$-decomposition.

\begin{lemma}\label{approxintoexact} Let $r\geq 3$ and let $G$ be a graph for which there is a set of non-negative weights $w_K$, $K\in \KK_{2r+2}$, such that, for each $e\in E(G)$,
\begin{equation}\label{important}
1-\frac 1{2r+1}\leq \sum_{K\in \KK_{2r+2}:e\in E(K)}w_K\leq 1.
\end{equation}
Then, $G$ has a fractional $K_r$-decomposition.
\end{lemma}

We then use this lemma to deduce from Lemma~\ref{fraccliqueapprox1} that any graph $G$ with minimum degree at least $(1-1/(128r+252))|G|$ has a fractional $K_r$-decomposition (see Lemma~\ref{maintheoremweakened}).
In order to prove Lemma~\ref{approxintoexact}, we first show that we can weight the $r$-cliques of $K_{2r+2}$ to achieve any particular weight over each edge, as long as these weights lie in $[1-1/(2r+1),1]$. The proof of the following lemma was sketched in Section~\ref{sketch:correction}.

\begin{lemma}\label{correction} Let $r\geq 3$ and $f:E(K_{2r+2})\to [1-1/(2r+1),1]$. Let $\KK_r=\KK_r(K_{2r+2})$. Then, there is a set of non-negative weights $w_K$, $K\in \KK_r$, such that, for each $e\in E(K_{2r+2})$,
\[
\sum_{K\in \KK_r:e\in E(K)}w_K=f(e).
\]
\ifcutfinished \qed
\else
\fi
\end{lemma}
\ifcutfinished
\else
\begin{proof} As is well-known, we can find disjoint matchings $M_1,\ldots,M_{2r+1}$ such that $\bigcup_iM_i=E(K_{2r+2})$.
Let $\MM$ be the set of matchings $M\subset E(K_{2r+2})$. By Lemma~\ref{minusmatching}, for each $M\in\MM$, we can choose weights $w_{M,K}\geq 0$, $K\in \KK_r$, so that, for each $e\in E(K_{2r+2})$,
\begin{equation}\label{fri2}
\sum_{K\in \KK_r:e\in E(K)}w_{M,K}=\mathbf{1}_{\{e\notin M\}}.
\end{equation}

Now, for each $i\in[2r+1]$, label the edges in $M_i$ as $e_{i,1},\ldots,e_{i,r+1}$, so that $f(e_{i,1})\leq f(e_{i,2})\leq\ldots\leq f(e_{i,r+1})$. For each $M\in \MM$, let
\begin{equation}\label{finaleq}
w_{M}=\left\{
\begin{array}{ll}
f(e_{i,j+1})-f(e_{i,j}) &\text{ if }M=\{e_{i,1},\ldots,e_{i,j}\}\text{ for some }i\in[2r+1],j\in[r]
\\
1-f(e_{i,r+1}) & \text{ if }M=\{e_{i,1},\ldots,e_{i,r+1}\}\text{ for some }i\in[2r+1]
\\
1-\sum_{i\in [2r+1]}(1-f(e_{i,1})) & \text{ if }M=\emptyset
\\
0 & \text{ otherwise.}
\end{array}
\right.
\end{equation}
As $f(e_{i,1})\geq 1-1/(2r+1)$ for each $i\in [2r+1]$, we have $w_\emptyset\geq 0$. As $f(e_{i,r+1})\leq 1$, and due to the ordering of the edges in each matching $M_i$, the weights $w_M$, $M\in \MM$, give a non-negative weighting of the matchings in $\MM$. For each $K\in \KK_r$, let
\begin{equation}\label{fri1}
w_K=\sum_{M\in \MM}w_M\cdot w_{M,K},
\end{equation}
so that $w_K$, $K\in \KK_r$, is a non-negative weighting of $\KK_r$.

Note that, for each $i\in[2r+1]$ and $j\in [r+1]$, we have
\begin{equation}\label{fri3}
\sum_{j'< j}w_{\{e_{i,1},\ldots,e_{i,j'}\}}=f(e_{i,j})-f(e_{i,1}),
\end{equation}
and, for each $i\in[2r+1]$, we have
\begin{equation}\label{fri4}
\sum_{j\in[r+1]}w_{\{e_{i,1},\ldots,e_{i,j}\}}=1-f(e_{i,1}).
\end{equation}
Therefore, for each $e\in E(G)$, letting $i$ and $j$ be such that $e=e_{i,j}$, we have
\begin{align*}
\sum_{K\in \KK_r:e\in E(K)}w_K\hspace{0.1cm}&\hspace{-0.1cm}\overset{\eqref{fri1}}{=}\sum_{M\in \MM}w_M\cdot\sum_{K\in \KK_r:e\in E(K)}w_{M,K}
\overset{\eqref{fri2}}{=}\sum_{M\in \MM}w_M\cdot\mathbf{1}_{\{e\notin M\}} \\
\hspace{-2cm}&= w_{\emptyset}+\sum_{i'\neq i}\sum_{j'\in[r+1]}w_{\{e_{i',1},\ldots,e_{i',j'}\}}
+\sum_{j'< j}w_{\{e_{i,1},\ldots,e_{i,j'}\}}
\\
&\hspace{-0.45cm}\overset{\eqref{fri4},\eqref{fri3}}{=}w_\emptyset+\sum_{i'\neq i}(1-f(e_{i',1}))+f(e_{i,j})-f(e_{i,1})\\
&\hspace{-0.1cm}\overset{\eqref{finaleq}}{=} 1-(1-f(e_{i,1}))+f(e_{i,j})-f(e_{i,1})= f(e_{i,j}),
\end{align*}
as required.
\end{proof}
\fi

We can now prove Lemma~\ref{approxintoexact} by using Lemma~\ref{correction} on each $(2r+2)$-clique of $G$ simultaneously.

\begin{proof}[Proof of Lemma~\ref{approxintoexact}.]
\ifcutfinished
\else
For each $e\in E(G)$, let $z_e=(1+1/(2r))\cdot\sum_{K\in \KK_{2r+2}:e\in E(K)}w_K$, so that, using~\eqref{important}, we have $1-1/(2r+1)\leq 1/z_e \leq 1$. For each $K\in \KK_{2r+2}$, using Lemma~\ref{correction}, take non-negative weights $w_{K,K'}$, $K'\in \KK_r=\KK_r(G)$, so that, for each $e\in E(G)$,
\begin{equation}\label{referback}
\sum_{K'\in \KK_r:e\in E(K')}w_{K,K'}=(1/z_e)\cdot \mathbf{1}_{\{e\in E(K)\}}.
\end{equation}

For each $K'\in \KK_r$, let $w_{K'}=\sum_{K\in \KK_{2r+2}}(1+1/(2r))w_{K}\cdot w_{K,K'}$, so that we have a non-negative weighting of $\KK_r$. Then, for each $e\in E(G)$, we have
\begin{align*}
\sum_{K'\in \KK_r:e\in E(K')}w_{K'}&=\sum_{K\in \KK_{2r+2}}(1+1/(2r))w_{K}\cdot\sum_{K'\in \KK_r:e\in E(K')}w_{K,K'}
\\
&\hspace{-0.12cm}\overset{\eqref{referback}}{=}\sum_{K\in \KK_{2r+2}}(1+1/(2r))w_K\cdot (1/z_e)\cdot \mathbf{1}_{\{e\in E(K)\}}
\\
&=\Big((1+1/(2r)) \sum_{K\in \KK_{2r+2}:e\in E(K)}w_K\Big)/z_e=1.
\end{align*}
Thus, the weights $w_{K'}$, $K'\in \KK_r$, form a fractional $K_r$-decomposition of $G$.
\fi
\end{proof}

We can now conclude from Lemma~\ref{fraccliqueapprox1} and Lemma~\ref{approxintoexact} the following weakened form of Theorem~\ref{maintheorem}.

\begin{lemma}\label{maintheoremweakened}
Let $r\geq 3$. If a graph $G$ has minimum degree at least $(1-1/(128r+252))|G|$, then $G$ has a fractional $K_r$-decomposition.
\end{lemma}

\begin{proof}
\ifcutfinished
\else Let $m=|G|$ and form a graph $G'$ by copying each vertex of $G$ $64r+126$ times, where two vertices in $G'$ have an edge between them if and only if the two original vertices did in $G$. Then, we have $\delta(G')\geq |G'|-(64r+126)|G|/(128r+252)=|G'|-m/2$.

By Lemma~\ref{fraccliqueapprox1}, there is a set of weights $w_K$, $K\in \KK_{2r+2}:= \KK_{2r+2}(G')$, such that, for each $e\in E(G')$,
\begin{equation*}\label{approxweights1}
1-\frac 1{2r+2}\leq \sum_{K\in \KK_{2r+2}:e\in E(K)}w_K\leq 1.
\end{equation*}
Thus, by Lemma~\ref{approxintoexact}, $G'$ has a fractional $K_r$-decomposition.

Each $r$-clique in $G'$ has vertices which are copied from some shared $r$-clique in $G$. Let $\tau:\KK_r(G')\to\KK_r(G)$ be the natural projection so that, for each $K\in \KK_r(G')$, $V(K)$ is a set of copies of vertices in $V(\tau(K))$. For each $K\in \KK_r(G)$, letting $w_K=(1/k^2)\cdot\sum_{K'\in \KK_r(G'):\tau(K')=K}w_{K'}$, with $k=64r+126$, then forms a fractional $K_r$-decomposition of $G$.
\fi
\end{proof}


\section{Improving on Lemma~\ref{fraccliqueapprox1}}\label{improvements}
To improve on Lemma~\ref{fraccliqueapprox1}, we essentially consider the following question: For how small a value of $\ell$ can we show that any graph $G$ with $2\ell$ vertices containing a copy of $M_{\ell}$ has a $K_r$-decomposition (see Definition~\ref{Mrdefn})? For Lemma~\ref{fraccliqueapprox1}, where we used $\ell\geq 16r+31$, we simply used Lemma~\ref{minusmatching} applied several times (via Lemma~\ref{minusmatchingind}). In this section, we will improve this to use $\ell\geq 9r+8$ in Lemma~\ref{Mdecomp} (taking also that $\ell$ is even to simplify the proof). We will first define a new graph $W_k$ and show that, for sufficiently large $k$, any graph with a spanning copy of $W_k$ has a fractional $K_r$-decomposition.


\begin{definition} \changed{Let $W_k$ be a complete $k$-partite graph with classes of size $4$ on the vertex set $[4k]$.}
\end{definition}

\begin{lemma}\label{Wfracdecomp} Let $r\geq 3$, $k\geq (3r+2)/2$, and let $G$ be a graph with $4k$ vertices which contains a copy of $W_{k}$. Then, $G$ has a fractional $K_r$-decomposition.
\ifcutfinished \qed
\else
\fi
\end{lemma}
\ifcutfinished
\else
\begin{proof} Using that $G$ contains a copy of $W_k$, divide the set $V(G)$ into $A_1,\ldots,A_{k}$ so that each set $A_i$ has size 4 and $G$ contains all the edges between different sets $A_i$. Let $E\subset E(G)$ be the set of edges of $G$ which lie within some set $A_i$. Let $\MM$ be the set of all subgraphs $M\subset G$ with $|M|=2r+2$ vertices and $\delta(M)\geq |M|-2$. By picking $M\in \MM$ randomly, we will weight the subgraphs in $\MM$ so that each edge in $G$ is given total weight $1$. As, by Lemma~\ref{minusmatching}, each subgraph $M\in \MM$ has a fractional $K_r$-decomposition, this implies that $G$ has a fractional $K_r$-decomposition.

Pick $I\subset [k]$ randomly so that $|I|=r+1$, and, starting with $V_0=\emptyset$, for each $i\in I$, pick uniformly and independently a random subset of $A_i$ with size 2 and add it to $V_0$. Let $M=G[V_0]$. Note that $M\in \MM$.

If $e\in E(G)\setminus E$ lies between $A_i$ and $A_j$, then
\[
\PP(e\in E(M))=\PP(i,j\in I)\cdot 1/4=\changed{\frac{\binom{r+1}{2}}{4\binom{k}{2}}},
\]
and, if $e\in E$ lies within $A_i$, then
\[
\PP(e\in E(M))=\PP(i\in I)\cdot 1/6=(r+1)/(6k)\geq \changed{\frac{\binom{r+1}{2}}{4\binom{k}{2}}},
\]
as $k\geq (3r+2)/2$.
Note that, if we delete any set of edges in $E$ from $M$, then the resulting subgraph is still in $\MM$. Therefore, we can alter our probability distribution to get a random subgraph $M'\in \MM$ so that, for every $e\in E(G)$, we have $\PP(e\in E(M'))=\binom{r+1}{2}/4\binom{k}{2}$. Thus, normalising appropriately, we can find weights $w_M$, $M\in \MM$, so that, for each $e\in E(G)$, $\sum_{M\in \MM:e\in E(M)}w_M=1$. As every graph in $\MM$ has a fractional $K_r$-decomposition, $G$ itself then has a fractional $K_r$-decomposition.
\end{proof}
\fi

We can now show that, if $\ell \geq (9r+8)/2$, then any graph containing a spanning copy of $M_{2\ell}$ has a fractional $K_r$-decomposition. We do this by fractionally decomposing such a graph into subgraphs with a spanning copy of $W_k$ or $\bar{C}_{4k}$ (for some appropriate $k$), where $C_{4k}$ is the cycle on $4k$ vertices. By using Lemma~\ref{Wfracdecomp} and Lemma~\ref{minusmatchingind} respectively to find a fractional $K_r$-decomposition of these subgraphs, we will get a fractional $K_r$-decomposition of the original graph.

\begin{lemma}\label{Mdecomp}
Let $r\geq 3$ and $\ell\geq (9r+8)/2$. If $G$ is a graph with $4\ell$ vertices which contains a copy of $M_{2\ell}$, then $G$ has a fractional $K_r$-decomposition.
\ifcutfinished \qed
\else
\fi
\end{lemma}
\ifcutfinished
\else
\begin{proof}
Let $k=\lceil (3r+2)/2\rceil$. Label $V(G)$ as $V=\{a_1,b_1,\ldots,a_{2\ell},b_{2\ell}\}$ so that $G$ has all edges between sets $\{a_i,b_i\}$ and $\{a_j,b_j\}$ if $i\notin \{j-1,j,j+1\}$, \changed{which is} possible as $G$ contains a copy of $M_{2\ell}$. Let $\MM$ be the set of all subgraphs $H\subset G$ with $|H|=4k$ which contain a copy of either $\bar{C}_{4k}$ or $W_{k}$.

Note that, for any graph $H$ with $4k\ge 4r+6$ vertices which contains a copy of $\bar{C}_{4k}$, the set $E(\bar{H})$ can be covered by two matchings. Therefore, by Lemma~\ref{minusmatchingind}, any such graph $H$ has a fractional $K_r$-decomposition. In combination with Lemma~\ref{Wfracdecomp}, then, each graph in $\MM$ has a fractional $K_r$-decomposition.

For each $j\in [\ell]$, let $A_j=\{a_{2j-1},b_{2j-1},a_{2j},b_{2j}\}$.
For each $i\in \{0,1,2\}$, pick a random induced subgraph $C_i$ of $G$ with $|C_i|=4k$ according to the following procedure:
\begin{itemize}
\item Select a subset $I\subset [\ell]$ with $|I|=2k$ uniformly and independently.
\item For each $j\in I$,
\begin{itemize}
\item if $i=0$, let $V(C_{i})\cap A_j= \{a_{2j-1},b_{2j-1}\}$ or $\{a_{2j},b_{2j}\}$ with probability $1/2$,
\item if $i=1$, let $V(C_{i})\cap A_j= \{a_{2j-1},a_{2j}\}$ or $\{b_{2j-1},b_{2j}\}$ with probability $1/2$, and
\item if $i=2$, let $V(C_{i})\cap A_j= \{a_{2j-1},b_{2j}\}$ or $\{b_{2j-1},a_{2j}\}$ with probability $1/2$.
\end{itemize}
\item For each \changed{$j\notin I$}, let $V(C_{i})\cap A_j=\emptyset$.
\end{itemize}
Then, pick $i\in \{0,1,2\}$ uniformly and independently, and let $C=C_{i}$. Note that $C_0$ contains a copy of $W_k$ and $C_1$ and $C_2$ both contain a copy of $\bar{C}_{4k}$. Therefore, $C\in \MM$.

Let $E\subset E(G)$ be the set of edges of $G$ which each lie within some set $A_i$. If $e\in E(G)\setminus E$ lies between $A_i$ and $A_j$, then
\[
\PP(e\in E(C))=\PP(i,j\in I)\cdot 1/4=\changed{\frac{\binom{2k}{2}}{4\binom{\ell}{2}}},
\]
and if $e\in E$ lies within $A_i$, then
\[
\PP(e\in E(C))=\PP(i\in I)\cdot 1/6=k/(3\ell)\geq \changed{\frac{\binom{2k}{2}}{4\binom{\ell}{2}}}.
\]
Note that deleting any set of edges in $E$ from $C$ gives another subgraph in $\MM$. Therefore, we can alter our probability distribution to give a random subgraph $C'\in \MM$ so that $\PP(e\in E(C'))=\binom{2k}{2}/4\binom{\ell}{2}$.

That is, normalising appropriately, we can find weights $w_M$, $M\in \MM$, so that, for each $e\in E(G)$, we have $\sum_{M\in \MM:e\in E(M)}w_M=1$. Recalling that every graph in $\MM$ has a fractional $K_r$-decomposition, $G$ then has a fractional $K_r$-decomposition.
\end{proof}
\fi

We can now use this improved fractional decomposition of graphs containing a spanning copy of $M_{2\ell}$ to improve on Lemma~\ref{fraccliqueapprox1}.

\begin{lemma}\label{fraccliqueapprox2} Let $r\geq 3$, $\ell=\lceil (9r+8)/2\rceil$ and let $G$ be a graph with $n=4\ell m$ vertices and $\delta(G)\geq n-m/2$. Then, there is a set of weights $w_K$, $K\in \KK_r$, such that, for each edge $e\in E(G)$,
\begin{equation}\label{approxbound2}
1-\frac 1r\leq \sum_{K\in \KK_r:e\in E(K)}w_K\leq 1.
\end{equation}
\end{lemma}
\begin{proof} 
Let $\mathcal{M}$ be the set of induced subgraphs of $G$ with $4\ell$ vertices which contain a copy of $M_{2\ell}$. By Lemma~\ref{pickrandomsubgraph}, we can find non-negative weights $p_M$, $M\in \MM$, so that, for each $e\in E(G)$,
\begin{equation*}\label{edgeaim2}
1-\frac{1}{r}\leq 1-\frac{4}{2\ell}\leq \sum_{M\in \MM:e\in E(M)}p_M\leq 1.
\end{equation*}
By Lemma~\ref{Mdecomp}, each graph in $\MM$ has a fractional $K_r$-decomposition. Similarly to the proof of \changed{Lemma~\ref{fraccliqueapprox1}}, we can combine these fractional $K_r$-decompositions with the weighting of the graphs in $\MM$ to get a weighting $w_K$, $K\in \KK_r$, which satisfies~\eqref{approxbound2}.
\end{proof}

Subject only to the remaining proof of Lemma~\ref{minusmatching}, we can now complete the proof of Theorem~\ref{maintheorem}.

\begin{proof}[Proof of Theorem~\ref{maintheorem}.]
\ifcutfinished
\else
Let $\ell=\lceil (18r+26)/2\rceil$.
Note that we can assume $n=4\ell m\leq (36r+54)m\leq 50rm$ by copying each vertex $4\ell$ times. Then, we have $\delta(G)\geq |G|-m/2$. By Lemma~\ref{fraccliqueapprox2}, there is a set of weights $w_K$, $K\in \KK_{2r+2}$, such that, for each edge $e\in E(G)$,
\begin{equation*}\label{approxweights2}
1-\frac 1{2r+2}\leq \sum_{K\in \KK_{2r+2}:e\in E(K)}w_K\leq 1.
\end{equation*}
Thus, by Lemma~\ref{approxintoexact}, $G$ has a fractional $K_r$-decomposition.
\fi
\end{proof}


\section{Proof of Lemma~\ref{minusmatching}}\label{tedium}
It remains only to show that any complete graph $K_k$ on $k\geq 2r+2$ vertices with any set of independent edges $M$ removed, $K_k-M$, has a fractional $K_r$-decomposition, thus proving Lemma~\ref{minusmatching}. We will easily be able to show that we may assume that $k=2r+2$ and $1\leq |M|\leq r$. Then, up to symmetry, there are three types of edges in $K_{2r+2}-M$: edges with 0, 1, or 2 vertices in edges in $M$. We will consider three types of $r$-clique in $K_k-M$ defined by the number of vertices they contain in $M$. By adding weight uniformly to each clique of one type we can add weight to the edges of $K_k-M$ so that the same weight is added to edges of the same type. By choosing three different weights to add to the three different types of cliques, we can control the amount of weight added to each type of edge and gain a fractional $K_r$-decomposition of $K_k-M$.
\begin{proof}[Proof of Lemma~\ref{minusmatching}]
Note that we can assume that $G$ has $2r+2$ vertices. Indeed, if $G$ has more than $2r+2$ vertices, then by giving the induced $2r+2$ vertex subgraphs of $G$ an appropriate uniform weight we can fractionally decompose $G$ into graphs with $2r+2$ vertices each of which only lacks edges in some matching.

Let $G$ then be a graph with vertex set $[2r+2]$ so that $M:=E(\bar{G})$ is a matching. Let $k=|M|$, and let $A$ be the set of vertices in some edge in $M$, so that $|A|=2k$.
If $k=r+1$ or $0$, \changed{then adding weight $1$ to every $r$-clique in $G$ weights the edges of $G$ uniformly  (due to the symmetry in $G$), and hence, by normalising these weights appropriately, we can} find a fractional $K_r$-decomposition of $G$. Let us assume then that $1\leq k\leq r$.

For each $i\in \{0,1,2\}$, let $E_i=\{e\in E(G):|V(e)\cap A|=i\}$. If we can find a non-negative weighting $w_K$, $K\in \KK_r$, so that, for each $i\in \{0,1,2\}$,
\begin{equation}\label{extra}
\sum_{K\in \KK_r}|E(K)\cap E_i|\cdot w_K=|E_i|,
\end{equation}
then, by the symmetry in $G$, we can easily convert this into a fractional $K_r$-decomposition of $G$. {Indeed, given such a weighting $w_K$, $K\in \KK_r$, for each $K\in \KK_r$ let
\begin{equation}\label{extra3}
w'_K=\frac{1}{|\{K'\in \KK_r:|V(K')\cap A|=|V(K)\cap A|\}|}\sum_{K'\in \KK_r:|V(K')\cap A|=|V(K)\cap A|}w_{K'}.
\end{equation}
Note that, for each $\ell\in \{0,1,\ldots,k\}$, $w'_K$ is the same for each $K\in \KK_r$ with $|V(K)\cap A|=\ell$. Due to the symmetry in $G$, then, for each $i\in \{0,1,2\}$, $z_e:=\sum_{K\in \KK_r:e\in E(K)}w'_K$ is the same, $z_i$ say, for each $e\in E_i$. Furthermore, for each $\ell\in \{0,1,\ldots,k\}$ and $i\in \{0,1,2\}$, $|E_i\cap E(K)|$ is the same for each $K\in \KK_r$ with $|V(K)\cap A|=\ell$.
Thus, for each $\ell\in \{0,1,\ldots,k\}$ and $i\in \{0,1,2\}$, we have from \eqref{extra3} that
\begin{equation}\label{extra2}
\sum_{K\in \KK_r:|V(K)\cap A|=\ell}|E_i\cap E(K)|\cdot w'_K=\sum_{K\in \KK_r:|V(K)\cap A|=\ell}|E_i\cap E(K)|\cdot w_{K}.
\end{equation}
Therefore, for each $i\in \{0,1,2\}$,
\[
|E_i|\cdot z_i=\sum_{e\in E_i}z_e=\sum_{e\in E_i}\sum_{K\in \KK_r:e\in E(K)}w'_K=\sum_{K\in \KK_r}|E(K)\cap E_i|\cdot w'_K\overset{\eqref{extra2}}{=}\sum_{K\in \KK_r}|E(K)\cap E_i|\cdot w_K\overset{\eqref{extra}}{=}|E_i|,
\]
and hence $z_i=1$. That is, for any $e\in E(G)$, $\sum_{K\in \KK_r:e\in E(K)}w'_K=1$. Thus, it is sufficient to find a non-negative weighting that satisfies \eqref{extra}.}

For each $k'$, $0\leq k'\leq k$, if an $r$-clique has $k'$ vertices in $A$, then it contains $k'(k'-1)/2$ edges in $E_2$, $k'(r-k')$ edges in $E_1$ and $(r-k')(r-k'-1)/2$ edges in $E_0$. Furthermore, $|E_2|=2k(2k-2)/2$, $|E_1|=2k(2r+2-2k)$ and $|E_0|=(2r+2-2k)(2r+1-2k)/2$. Note that there exists an $r$-clique in $G$ with $k'$ vertices in $A$ if $0\le k'\le k$ and $r-k'\leq 2r+2-|A|=2r+2-2k$, or, equivalently, $\max\{0,2k-r-2\}\le k'\le k$.

Thus, to prove the lemma it is sufficient to find $\max\{0,2k-r-2\}\le k_1,k_2,k_3\le k$ and $x,y,z\geq 0$ for which we have
\[
\begin{pmatrix}
\frac12 k_1(k_1-1) &  \frac12 k_2(k_2-1) & \frac12 k_3(k_3-1) \\
k_1(r-k_1) & k_2(r-k_2) & k_3(r-k_3) \\
\frac12 (r-k_1)(r-k_1-1) & \frac12 (r-k_2)(r-k_2-1) & \frac12 (r-k_3)(r-k_3-1)
\end{pmatrix}
\cdot
\begin{pmatrix} x \\ y \\ z
\end{pmatrix}
\]
\begin{equation}\label{modnasty}\hspace{6.5cm}
=
\begin{pmatrix} k(2k-2) \\ 2k(2r+2-2k) \\ (r+1-k)(2r+1-2k)
\end{pmatrix}.
\end{equation}

Now, if $2k-r-2\ge 0$, then taking $k_1=k$, $k_2=k-1$, $k_3=2k-r-2$,
\[
x=\frac{2(r+1-k)(2(r-k)^2+kr+5r-4k+2)}{r(r-1)(r+2-k)},\;\;
\]
\[
y=\frac{2k(3r-2k+2)}{r(r-1)}\text{ and }
z=\frac{2k(k-1)}{r(r-1)(r+2-k)}
\]
satisfies~\eqref{modnasty}. As $k\leq r$, we have $k-1>2k-r-2$, so that $2k-r-2\leq k_1,k_2,k_3\leq r$ and, as $r\geq 3$ and $1\leq k\leq r$, we have $x,y,z\ge 0$.

If $2k-r-2<0$, then taking $k_1=k$, $k_2=k-1$, $k_3=0$,
\[
x=\frac{4(r+1-k)}{r-1},\;\;
y=\frac{4k}{r-1},\text{ and }
z=\frac{2(r+1-k)}{r(r-1)}
\]
satisfies~\eqref{modnasty}, where $x,y,z\ge 0$ as $k\le r$ and $r\ge 3$. Thus, for all possible values of $k$, $G$ has a fractional $K_r$-decomposition.
\end{proof}




\end{document}

%% file: processpicture.tex
\hspace{-1.1cm}
{\scalefont{0.4}
\begin{tikzpicture}[scale=0.7]


\foreach \x in {1,2,3,4,7}
{
\draw ({-1+3*\x},5.6) -- ({-1+3*\x},7.5);
\draw ({3*\x},5.6) -- ({3*\x},7.5);
\draw  ({3*\x},7.5)
arc [start angle=0, end angle=180,
x radius=0.5cm, y radius=0.5cm];
\draw  ({-1+3*\x},5.6)
arc [start angle=-180, end angle=0,
x radius=0.5cm, y radius=0.5cm];

\draw ({-1+3*\x},2.5) -- ({-1+3*\x},4.4);
\draw ({3*\x},2.5) -- ({3*\x},4.4);
\draw  ({3*\x},4.4)
arc [start angle=0, end angle=180,
x radius=0.5cm, y radius=0.5cm];
\draw  ({-1+3*\x},2.5)
arc [start angle=-180, end angle=0,
x radius=0.5cm, y radius=0.5cm];

\draw ({-1.1+3*\x},2.5) -- ({-1.1+3*\x},7.5);
\draw ({3*\x+0.1},2.5) -- ({3*\x+0.1},7.5);
\draw  ({3*\x+0.1},7.5)
arc [start angle=0, end angle=180,
x radius=0.6cm, y radius=0.6cm];
\draw  ({-1.1+3*\x},2.5)
arc [start angle=-180, end angle=0,
x radius=0.6cm, y radius=0.6cm];
}

\foreach \x in {6}
{
\foreach \y in {-2,-1,0,1}
{
\draw [fill] ({3*\x+0.5* \y-0.75-0.5},5) circle [radius=0.05cm];
}
}

\foreach \x in {1,2,3,4}
{
\draw ({3*\x-0.5},8.4) node {\large $A_{\x}$};
\draw ({3*\x-0.9},0) node {\large $a_{\x}$};
\draw ({3*\x+0.05},0) node {\large $b_{\x}$};

\draw [fill] ({3*\x-0.9},0.5) circle [radius=0.075cm];
\draw [fill] ({3*\x-0.1},0.5) circle [radius=0.075cm];
}
\foreach \x in {7}
{
\draw ({3*\x-0.5},8.4) node {\large $A_{r}$};
}

\foreach \x in {6}
{
\draw ({3*\x-1},0) node {\large $a_{r-1}$};
\draw ({3*\x+0.1},0) node {\large $b_{r-1}$};

\draw [fill] ({3*\x-0.9},0.5) circle [radius=0.075cm];
\draw [fill] ({3*\x-0.1},0.5) circle [radius=0.075cm];
}

\foreach \x in {7}
{
\draw ({3*\x-0.9},0) node {\large $a_{r}$};
\draw ({3*\x+0.05},0) node {\large $b_{r}$};

\draw [fill] ({3*\x-0.9},0.5) circle [radius=0.075cm];
\draw [fill] ({3*\x-0.1},0.5) circle [radius=0.075cm];
}

\draw ({3*1-1.2},{-0.65}) -- ({3*7+0.2},{-0.65});
\draw  ({3*7+0.2},-0.65)
arc [start angle=-90, end angle=0,
x radius=0.25cm, y radius=0.25cm];
\draw  ({3*1-1.2},-0.65)
arc [start angle=+270, end angle=180,
x radius=0.25cm, y radius=0.25cm];

\draw ({3*4-0.6},-0.65) -- ({3*4-0.6},-0.85);
\draw ({3*4-0.6},-1.25) node {\large $B=V(M)$};

\foreach \x in {1,4}
{
\draw [fill] ({3*\x-0.5},2.5) circle [radius=0.075cm];
\draw [fill] ({3*\x-0.5},3) circle [radius=0.075cm];

\draw  ({3*\x-0.5+0.2},2.75) to [out=0,in=90] ({3*\x+0.5},2);
\draw [->] ({3*\x+0.5},2) to [out=270,in=45] ({3*\x-0.4},0.7);
\draw ({3*\x-0.4},0.7) -- ({3*\x-0.4},0.85);
\draw ({3*\x-0.4},0.7) -- ({3*\x-0.4+0.15},0.7);

}
\foreach \x in {2}
{
\draw [fill] ({3*\x-0.5},4.4) circle [radius=0.075cm];
\draw [fill] ({3*\x-0.5},5.6) circle [radius=0.075cm];

\draw [->] ({3*\x-0.5+0.2},5) to [out=0,in=90] ({3*\x+1},3.5)
          to [out=270,in=45] ({3*\x-0.4},0.7);

\draw ({3*\x-0.4},0.7) -- ({3*\x-0.4},0.85);
\draw ({3*\x-0.4},0.7) -- ({3*\x-0.4+0.15},0.7);
}
\foreach \x in {3,7}
{
\draw [fill] ({3*\x-0.5},6.3) circle [radius=0.075cm];
\draw [fill] ({3*\x-0.5},6.8) circle [radius=0.075cm];

\draw [->] ({3*\x-0.5+0.2},6.55) to [out=0,in=90] ({3*\x+0.5+0.5},4)
          to [out=270,in=45] ({3*\x-0.4},0.7);

\draw ({3*\x-0.4},0.7) -- ({3*\x-0.4},0.85);
\draw ({3*\x-0.4},0.7) -- ({3*\x-0.4+0.15},0.7);
}

\draw (1.2,3.45) node {\large $A_{1,1}$};
\draw (1.7,3.45) -- (2,3.45);
\draw (1.7,6.55) -- (2,6.55);
\draw (1.2,6.55) node {\large $A_{1,2}$};

\draw (19.2,3.45) node {\large $A_{r,1}$};
\draw (19.7,3.45) -- (20,3.45);
\draw (19.7,6.55) -- (20,6.55);
\draw (19.2,6.55) node {\large $A_{r,2}$};


\foreach \be in {dotted}
{
\foreach \gr in {gray}
{
\foreach \x in {2,5}
{
\draw [\be] ({-0.9+3*\x},2.5) -- ({-0.9+3*\x},3.5);
\draw [\be] ({3*\x-0.1},2.5) -- ({3*\x-0.1},3.5);
\draw [\be] ({3*\x-0.1},3.5)
arc [start angle=0, end angle=180,
x radius=0.4cm, y radius=0.4cm];
\draw [\be] ({-1+3*\x+0.1},2.5)
arc [start angle=-180, end angle=0,
x radius=0.4cm, y radius=0.4cm];

\draw [dotted] ({3*(\x-1)-0.9},0.5) -- ({-0.85+3*\x},3.7);
\draw [dotted] ({3*(\x-1)-0.1},0.5) -- ({-0.5+3*\x},2.1);

\draw [\gr] ({3*\x-0.5},3) node {\large $B_\x$};
}

\foreach \x in {3}
{
\draw [\be] ({-0.9+3*\x},2.5) -- ({-0.9+3*\x},4);
\draw [\be] ({3*\x-0.1},2.5) -- ({3*\x-0.1},4);
\draw [\be]  ({3*\x-0.1},4)
arc [start angle=0, end angle=180,
x radius=0.4cm, y radius=0.4cm];
\draw [\be]  ({-1+3*\x+0.1},2.5)
arc [start angle=-180, end angle=0,
x radius=0.4cm, y radius=0.4cm];

\draw [dotted] ({3*(\x-1)-0.9},0.5) -- ({-0.85+3*\x},4.2);
\draw [dotted] ({3*(\x-1)-0.1},0.5) -- ({-0.5+3*\x},2.1);

\draw [\gr] ({3*\x-0.5},3.25) node {\large $B_\x$};
}

\foreach \x in {4}
{
\draw [\be] ({-0.9+3*\x},2.5) -- ({-0.9+3*\x},3);
\draw [\be] ({3*\x-0.1},2.5) -- ({3*\x-0.1},3);
\draw [\be]  ({3*\x-0.1},3)
arc [start angle=0, end angle=180,
x radius=0.4cm, y radius=0.4cm];
\draw [\be]  ({-1+3*\x+0.1},2.5)
arc [start angle=-180, end angle=0,
x radius=0.4cm, y radius=0.4cm];

\draw [dotted] ({3*(\x-1)-0.9},0.5) -- ({-0.85+3*\x},3.2);
\draw [dotted] ({3*(\x-1)-0.1},0.5) -- ({-0.5+3*\x},2.1);

\draw [\gr] ({3*\x-0.5},3.8) node {\large $B_\x$};
\draw [\be] ({3*\x-0.5},3.4)--({3*\x-0.5},3.6);
}

\foreach \x in {7}
{
\draw [\be] ({-0.9+3*\x},2.5) -- ({-0.9+3*\x},3);
\draw [\be] ({3*\x-0.1},2.5) -- ({3*\x-0.1},3);
\draw  [\be] ({3*\x-0.1},3)
arc [start angle=0, end angle=180,
x radius=0.4cm, y radius=0.4cm];
\draw  [\be] ({-1+3*\x+0.1},2.5)
arc [start angle=-180, end angle=0,
x radius=0.4cm, y radius=0.4cm];

\draw [dotted] ({3*(\x-1)-0.9},0.5) -- ({-0.85+3*\x},3.2);
\draw [dotted] ({3*(\x-1)-0.1},0.5) -- ({-0.5+3*\x},2.1);

\draw [\gr] ({3*\x-0.5},2.75) node {\large $B_r$};
}

}
}

\end{tikzpicture}
}